\documentclass{amsart}

\makeatletter
\@namedef{subjclassname@2020}{%
  \textup{2020} Mathematics Subject Classification}
  
\makeatother 

\usepackage{amsthm,amssymb,amsfonts,latexsym,mathtools,thmtools,amsmath,lscape}
\usepackage[T1]{fontenc}
\usepackage{tikz-cd} 
\usepackage{enumitem} 
\usepackage{longtable}
\usepackage{mathrsfs} 
\usepackage{hyperref} 
\usepackage{hyperref}
\usepackage[all]{xy}
\usepackage{booktabs}
\hypersetup{
    colorlinks=true,
    linkcolor=blue,
    filecolor=blue,      
    urlcolor=blue,
    linktocpage=true
}

\newtheorem{theorem}{Theorem}[section]
\newtheorem{lemma}[theorem]{Lemma}
\newtheorem{proposition}[theorem]{Proposition}

\theoremstyle{definition}
\newtheorem{definition}[theorem]{Definition}

\newtheorem{remark}[theorem]{Remark}

\newtheorem{example}[theorem]{Example}

\theoremstyle{remark}

\numberwithin{equation}{section}

\begin{document}

\title[On the Differential Smoothness of Skew PBW Extensions] {A note on the differential smoothness \\ of skew PBW extensions}


\author{Andr\'es Rubiano}
\address{Universidad Nacional de Colombia - Sede Bogot\'a}
\curraddr{Campus Universitario}
\email{arubianos@unal.edu.co}
\address{Universidad ECCI}
\curraddr{Campus Universitario}
\email{arubianos@ecci.edu.co}
\thanks{}


\author{Armando Reyes}
\address{Universidad Nacional de Colombia - Sede Bogot\'a}
\curraddr{Campus Universitario}
\email{mareyesv@unal.edu.co}

\thanks{This work was supported by Faculty of Science, Universidad Nacional de Colombia - Sede Bogot\'a, Colombia [grant number 53880].}



\date{}

\dedicatory{Dedicated to Oswaldo Lezama on the Occasion of His 68th Birthday}

\begin{abstract} 

We investigate the differential smoothness of a certain family of skew Poincar\'e-Birkhoff-Witt extensions.

\end{abstract}

\maketitle


\section{Introduction}

Ore \cite{Ore1931, Ore1933} introduced a kind of noncommutative polynomial rings which has become one of most basic and useful constructions in ring theory and noncommutative algebra. For an associative and unital ring $R$, an endomorphism $\sigma$ of $R$ and a $\sigma$-derivation $\delta$ of $R$, the {\em Ore extension} or {\em skew polynomial ring} of $R$ is obtained by adding a single generator $x$ to $R$ subject to the relation $xr = \sigma(r) x + \delta(r)$ for all $r\in R$. This Ore extension of $R$ is denoted by $R[x; \sigma, \delta]$. As one can appreciate in the literature, a lot of papers and books have been published concerning ring-theoretical, homological, geometrical properties and applications of these extensions (e.g. \cite{BrownGoodearl2002, BuesoTorrecillasVerschoren2003, Fajardoetal2024, GoodearlLetzter1994, GoodearlWarfield2004, McConnellRobson2001, Li2002, SeilerBook2010} and references therein).

One of the topics of research on Ore extensions has been carried out by several authors concerning to the ring-theoretical notions of {\em Baer}, {\em quasi-Baer}, {\em p.p.}, {\em p.q.}, {\em Armendariz} and {\em reduced} rings (e.g. \cite{Birkenmeieretal2001b, Hashemietal2003, HashemiMoussavi2005, HongKimKwak2000, HongKimKwak2003, Huhetal2002, Kaplansky1968, Krempa1996, LeeWong2003, Matczuk2004} and references therein; for a detailed description of each one of these notions, see the excellent treatment developed by Birkenmeier et al. \cite{BirkenmeierParkRizvi2013}). All these properties are formulated to the ring of coefficients $R$ of the Ore extension $R[x; \sigma, \delta]$ under certain adequate conditions on the maps $\sigma$ and $\delta$. The {\em extension} of these theoretical notions from $R$ to $R[x; \sigma, \delta]$ was considered by Nasr-Isfahani and Moussavi \cite{NasrMoussavi2008}. In their paper, for a ring $R$ with an automorphism $\sigma$ and $\sigma$-derivation $\delta$ where $\alpha \delta = \delta \alpha$, they considered the $\overline{\sigma}$-derivation $\overline{\delta}$ on the Ore extension $R[x; \sigma, \delta]$ which, precisely, {\em extends} $\delta$. More exactly, for an element $f(x) = r_0 + r_1 x + \dotsb + r_nx^n \in R[x; \sigma, \delta]$ they defined the automorphism $\overline{\sigma}$ and the $\overline{\sigma}$-derivation $\overline{\delta}$ on $R[x; \sigma, \delta]$ as
\begin{align*}
    \overline{\sigma}(f(x)) = &\ \sigma(r_0) + \sigma(r_1) x + \dotsb + \sigma(r_n)x^n, \quad {\rm and} \\
     \overline{\delta}(f(x)) = &\ \delta(r_0) + \delta(r_1) x + \dotsb + \delta(r_n)x^n, 
\end{align*}

respectively. Such an automorphism $\overline{\sigma}$ and derivation $\overline{\delta}$ of $R[x; \sigma, \delta]$ are called {\em extended automorphism} and {\em derivation of} $\delta$.

Related to the study of transfer of ring-theoretical properties from coefficient rings to noncommutative polynomial extensions over these rings, 
Kwak et al. \cite{Kwaketal2014} extended the reflexive property to the skewed reflexive property by ring endomorphisms. An endomorphism $\sigma$ of a ring $R$ is called \textit{right} (resp., \textit{left}) \textit{skew reflexive} if for
$a, b \in R$, $aRb = 0$ implies $bR\sigma(a) = 0$ (resp., $\sigma(b)Ra = 0$), and $R$ is called \textit{right} (resp., \textit{left}) \textit{$\sigma$-skew reflexive} if there exists a right (resp., left) skew reflexive endomorphism $\sigma$ of $R$. $R$ is said to be \textit{$\sigma$-skew reflexive} if it is both right and left $\sigma$-skew reflexive. It is clear that $\sigma$-rigid rings are right $\sigma$-skew reflexive. More precisely, a ring $R$ is reduced and right $\sigma$-skew reflexive for a monomorphism $\sigma$ of $R$ if and only if $R$ is $\sigma$-rigid \cite[Theorem 2.6]{Kwaketal2014}. Bhattacharjee \cite{Bhattacharjee2020} extend the notion of RNP rings (Kheradmand et al. \cite{Kheradmandetal2017}, where \textit{RNP} means {\em reflexive-nilpotents-property}) to ring endomorphisms $\sigma$ and introduced the notion of $\sigma$-skew RNP rings as a generalization of $\sigma$-skew reflexive rings. An endomorphism $\sigma$ of a ring $R$ is called \textit{right} (resp., \textit{left}) \textit{skew} RNP if for $a, b \in N(R)$ (the set of nilpotent elements of $R$), $aRb = 0$ implies $bR\sigma(a) = 0$ (resp., $\sigma(b)Ra = 0$). A ring $R$ is called \textit{right} (resp., \textit{left}) \textit{$\sigma$-skew} RNP if there exists a right (resp., left) skew RNP endomorphism $\sigma$ of $R$. $R$ is said to be \textit{$\sigma$-skew} RNP if it is both right and left $\sigma$-skew RNP. From \cite[Remark 1.2]{Bhattacharjee2020}, we know that reduced rings are $\sigma$-skew RNP for any endomorphism $\sigma$, and every right (resp., left) $\sigma$-skew reflexive ring is right (resp., left) $\sigma$-skew RNP. By \cite[Example 1.3]{Bhattacharjee2020}, we have that the notion of $\sigma$-skew RNP ring is not left-right symmetric. However, if $R$ is an RNP ring with an endomorphism $\sigma$, then $R$ is right $\sigma$-skew RNP if and only if $R$ is left $\sigma$-skew RNP. 

With the aim of generalizing Ore extensions of injective type (that is, $R[x; \sigma, \delta]$ with $\sigma$ an injective map) and other families of noncommutative rings appearing in the literature, Gallego and Lezama \cite{GallegoLezama2011} defined the {\em skew Poincar\'e-Birkhoff-Witt extensions} ({\em SPBW extensions} for short). Since then, ring, homological and geometrical properties of these objects have been studied by some authors (e.g. \cite{AbdiTalebi2023, Artamonov2015, Hamidizadehetal2020, HigueraReyes2023, LezamaReyes2014, NinoRamirezReyes2020, ReyesSuarez2020, ReyesSuarez2021Radicals}; a detailed treatment can be found in Fajardo et al \cite{Fajardoetal2020}). In particular, and due to the relation between Ore extensions and SPBW extensions in terms of endomorphisms and derivations (Proposition \ref{GallegoLezama2011Proposition3}), Reyes and Su\'arez \cite{ReyesSuarez2017} carried out a similar work to the presented by by Nasr-Isfahani and Moussavi \cite{NasrMoussavi2008} but now in the more general setting of SPBW extensions (see Section \ref{SPBWdefinitionspreliminaries} for all details), while Su\'arez et al. \cite{SuarezHigueraReyes2024} studied the reflexive-nilpotents-property for skew PBW extensions. There, they introduced the $\Sigma$-{\em skew CN} and $\Sigma$-{\em skew reflexive} ({\em RNP}) {\em rings} and under conditions of compatibility, they investigated the transfer of the reflexive-nilpotents-property from a ring of coefficients to a skew PBW extension over it. Their results extend those corresponding presented by Bhattacharjee \cite{Bhattacharjee2020} for Ore extensions.

On the other hand, the geometric notion of interest in this paper is {\em differential smoothness} in the Brzezi\'nski and Sitarz's sense \cite{BrzezinskiSitarz2017}. Before to define it, let us say some words on the origin of {\em smoothness} of algebras.

Just as Brzezi\'nski and Lomp said \cite[Section 1]{BrzezinskiLomp2018}, the study of this smoothness goes back at least to Grothendieck's EGA \cite{Grothendieck1964}. The concept of a {\em formally smooth commutative} ({\em topological}) {\em algebra} introduced by him was extended to the noncommutative setting by Schelter \cite{Schelter1986}. An algebra is {\em formally smooth} if and only if the kernel of the multiplication map is projective as a bimodule. This notion arose as a replacement of a far too general definition based on the finiteness of the global dimension; Cuntz and Quillen \cite{CuntzQuillen1995} called these algebras {\em quasi-free}. Precisely, the notion of smoothness based on the finiteness of this dimension was refined by Stafford and Zhang \cite{StaffordZhang1994}, where a Noetherian algebra is said to be {\em smooth} provided that it has a finite global dimension equal to the homological dimension of all its simple modules. In the homological setting, Van den Bergh \cite{VandenBergh1998} called an algebra {\em homologically smooth} if it admits a finite resolution by finitely generated projective bimodules. The characterization of this kind of smoothness for the noncommutative pillow, the quantum teardrops, and quantum homogeneous spaces was made by Brzezi{\'n}ski \cite{Brzezinski2008, Brzezinski2014} and Kr\"ahmer \cite{Krahmer2012}, respectively.

Brzezi{\'n}ski and Sitarz \cite{BrzezinskiSitarz2017} defined other notion of smoothness of algebras, termed {\em differential smoothness} due to the use of differential graded algebras of a specified dimension that admits a noncommutative version of the Hodge star isomorphism, which considers the existence of a top form in a differential calculus over an algebra together with a string version of the Poincar\'e duality realized as an isomorphism between complexes of differential and integral forms. This new notion of smoothness is different and more constructive than the homological smoothness mentioned above. \textquotedblleft The idea behind the {\em differential smoothness} of algebras is rooted in the observation that a classical smooth orientable manifold, in addition to de Rham complex of differential forms, admits also the complex of {\em integral forms} isomorphic to the de Rham complex \cite[Section 4.5]{Manin1997}. The de Rham differential can be understood as a special left connection, while the boundary operator in the complex of integral forms is an example of a {\em right connection}\textquotedblright\ \cite[p. 413]{BrzezinskiSitarz2017}.

Several authors (e.g. \cite{Brzezinski2015, Brzezinski2016, BrzezinskiElKaoutitLomp2010, BrzezinskiLomp2018, BrzezinskiSitarz2017, DuboisVioletteKernerMadore1990, Karacuha2015, KaracuhaLomp2014, ReyesSarmiento2022}) have characterized the differential smoothness of algebras such as the quantum two - and three - spheres, disc, plane, the noncommutative torus, the coordinate algebras of the quantum group $SU_q(2)$, the noncommutative pillow algebra, the quantum cone algebras, the quantum polynomial algebras, Hopf algebra domains of Gelfand-Kirillov dimension two that are not PI, families of Ore extensions, some 3-dimensional skew polynomial algebras, diffusion algebras in three generators, and noncommutative coordinate algebras of deformations of several examples of classical orbifolds such as the pillow orbifold, singular cones and lens spaces. An interesting fact is that some of these algebras are also homologically smooth in the Van den Bergh's sense.

With all above facts in mind, our purpose in this paper is to investigate the differential smoothness of skew PBW extensions considering extended automorphisms and derivations, from the ring of coefficients, to the skew PBW over this ring. In this way, we contribute to the the study of algebraic and geometric properties of the Ore extensions considered by Nasr-Isfahani and Moussavi \cite{NasrMoussavi2008} and Bhattacharjee \cite{Bhattacharjee2020}, and hence continue the research on the differential smoothness of noncommutative algebras related to SPBW extensions that have been carried out by the authors in \cite{ReyesSuarez2016, RubianoReyes2024DSBiquadraticAlgebras, RubianoReyes2024DSDoubleOreExtensions, RubianoReyes2024DSSPBWKt}.

The article is organized as follows. In Section \ref{DefinitionspreliminariesSPBWDS4} we review the key facts on SPBW extensions (Section \ref{SPBWdefinitionspreliminaries}) and differential smoothnes of algebras (Section \ref{DefinitionsandpreliminariesDSA}) in order to set up notation and render this paper self-contained. Section \ref{DSSPBW4originalresults} contains the original results of the paper. We start in Section \ref{ASPBWpreliminaries} with preliminary facts on the extension of automorphisms and derivations of a ring $R$ to a SPBW extension over $R$. Next, in Section \ref{DifferentialsmoothnesSPBWEAD} we characterize the differential smoothness of this kind of SPBW extensions (Theorem \ref{smoothSPBWn}).

\section{Definitions and preliminaries}\label{DefinitionspreliminariesSPBWDS4}

\subsection{Skew Poincar\'e-Birkhoff-Witt extensions}\label{SPBWdefinitionspreliminaries}

As it has been shown in the literature, SPBW extensions generalize several kinds of rings such as Ore extensions of injective type \cite{Ore1931, Ore1933}, PBW extensions \cite{BellGoodearl1988}, 3-dimensional skew polynomial algebras \cite{BellSmith1990}, diffusion algebras \cite{IsaevPyatovRittenberg2001, PyatovTwarock2002}, ambiskew polynomial \cite{Jordan2000}, solvable polynomial rings \cite{KandryWeispfenning1990}, almost normalizing extensions \cite{McConnellRobson2001}, skew bi-quadratic algebras, and some families of diskew polynomial rings and degree-one generalized Weyl algebras \cite{Bavula2020} or rank-one hyperbolic algebras \cite{Rosenberg1995}. Different relations between SPBW extensions and other noncommutative algebras having PBW bases can be find in \cite{BuesoTorrecillasVerschoren2003, GomezTorrecillas2014, Levandovskyy2005, Li2002, SeilerBook2010}). 

\begin{definition}[{\cite[Definition 1]{GallegoLezama2011}}]\label{defpbwextension} 
Let $R$ and $A$ be rings. We say that $A$ is a {\it SPBW extension} over $R$ if the following conditions hold:
\begin{itemize}
\item [(i)] $R$ is a subring of $A$ sharing the same identity element.

\item [(ii)] There exist elements $x_1, \ldots, x_n \in A\ \backslash\ R$ such that $A$ is a left free $R$-module with basis given by the set $\text{Mon}(A):= \{ x^{\alpha} = x_{1}^{\alpha_1} \cdots x_{n}^{\alpha_n}\mid \alpha =(\alpha_1, \ldots, \alpha_n) \in \mathbb{N}^n\}$.

\item[(iii)] For each $1 \leq i \leq n$ and any $r \in R\  \backslash\ \{0\}$, there exists an element $c_{i,r} \in R \ \backslash\ \{0\}$ such that $x_ir - c_{i,r}x_i \in R$.

\item[\rm (iv)]For $1\leq i, j\leq n$, there exists an element $d_{i,j}\in R\ \backslash\ \{0\}$ such that
\[
x_jx_i-d_{i,j}x_ix_j\in R+Rx_1+\cdots +Rx_n,
\]
i.e., there exist elements $r_0^{(i,j)}, r_1^{(i,j)}, \dotsc, r_n^{(i,j)} \in R$ with
\begin{center}\label{relSPBW}
$x_jx_i - d_{i,j}x_ix_j = r_0^{(i,j)} + \sum_{k=1}^{n} r_k^{(i,j)}x_k$.    
\end{center}
\end{itemize}
\end{definition}

We use freely the notation $A = \sigma(R)\langle x_1,\dotsc, x_n\rangle$ to denote a SPBW extension $A$ over a ring $R$ in the indeterminates $x_1, \dotsc, x_n$. $R$ is called the {\em ring of coefficients} of the extension $A$. Since $\text{Mon}(A)$ is a left $R$-basis of $A$, the elements $c_{i,r}$ and $d_{i,j}$ in Definition \ref{defpbwextension} are unique. Every element $f \in A\  \backslash\ \{0\}$ has a unique representation as $f = \sum_{i=0}^tr_iX_i$, with $r_i \in R\ \backslash\ \{0\}$ and $X_i \in \text{Mon}(A)$ for $0 \leq i \leq t$ with $X_0=1$. When necessary, we use the notation $f = \sum_{i=0}^tr_iY_i$. For $X=x^{\alpha}\in\text{Mon}(A)$, ${\rm exp}(X):=\alpha$ and $\deg(X):=|\alpha|$. If $f$ is an element as in (v), then $\deg(f):=\max\{\deg(X_i)\}_{i=1}^t$ \cite[Remark 2 and Definition 6]{GallegoLezama2011}.

The following proposition shows explicitly the relation between Ore extensions and SPBW extensions.

\begin{proposition}[{\cite[Proposition 3]{GallegoLezama2011}}]\label{GallegoLezama2011Proposition3}
If $A = \sigma(R)\langle x_1,\dotsc, x_n\rangle$ is a SPBW extension over $R$, then for each $1 \leq i \leq n$, there exist an injective endomorphism $\sigma_i : R \to R$ and a $\sigma_i$-derivation $\delta_i: R \to R$ such that $x_ir = \sigma_i(r)x_i + \delta_i(r)$, for each $r\in R$.  
\end{proposition}

We use the notation $\Sigma:=\{\sigma_1,\dots,\sigma_n\}$ and $\Delta:=\{\delta_1,\dots,\delta_n\}$ and say that the pair $(\Sigma, \Delta)$ is a \textit{system of endomorphisms and $\Sigma$-derivations} of $R$ with respect to $A$. For $\alpha = (\alpha_1, \dots , \alpha_n) \in \mathbb{N}^n$, $\sigma^{\alpha}:= \sigma_1^{\alpha_1}\circ \cdots \circ \sigma_n^{\alpha_n}$, $\delta^{\alpha} := \delta_1^{\alpha_1} \circ \cdots \circ \delta_n^{\alpha_n}$, where $\circ$ denotes the classical composition of functions. The system $\Sigma$ is commutative if $\sigma_i \circ \sigma_j = \sigma_j \circ \sigma_i$ for every $1 \le i, j \le n$. The commutativity for $\Delta$ is defined similarly. The system $(\Sigma, \Delta)$ is commutative if both $\Sigma$ and $\Delta$ are commutative \cite[Definition 2.1]{LezamaAcostaReyes2015}.

\subsection{Differential smoothness}\label{DefinitionsandpreliminariesDSA}

We follow Brzezi\'nski and Sitarz's presentation on differential smoothness carried out in \cite[Section 2]{BrzezinskiSitarz2017} (c.f. \cite{Brzezinski2008, Brzezinski2014}).

\begin{definition}[{\cite[Section 2.1]{BrzezinskiSitarz2017}}]
\begin{enumerate}
    \item [\rm (i)] A {\em differential graded algebra} is a non-negatively graded algebra $\Omega$ with the product denoted by $\wedge$ together with a degree-one linear map $d:\Omega^{\bullet} \to \Omega^{\bullet +1}$ that satisfies the graded Leibniz rule and is such that $d \circ d = 0$. 
    
    \item [\rm (ii)] A differential graded algebra $(\Omega, d)$ is a {\em calculus over an algebra} $A$ if $\Omega^0 A = A$ and $\Omega^n A = A\ dA \wedge dA \wedge \dotsb \wedge dA$ ($dA$ appears $n$-times) for all $n\in \mathbb{N}$ (this last is called the {\em density condition}). We write $(\Omega A, d)$ with $\Omega A = \bigoplus_{n\in \mathbb{N}} \Omega^{n}A$. By using the Leibniz rule, it follows that $\Omega^n A = dA \wedge dA \wedge \dotsb \wedge dA\ A$. A differential calculus $\Omega A$ is said to be {\em connected} if ${\rm ker}(d\mid_{\Omega^0 A}) = \Bbbk$.
    
    \item [\rm (iii)] A calculus $(\Omega A, d)$ is said to have {\em dimension} $n$ if $\Omega^n A\neq 0$ and $\Omega^m A = 0$ for all $m > n$. An $n$-dimensional calculus $\Omega A$ {\em admits a volume form} if $\Omega^n A$ is isomorphic to $A$ as a left and right $A$-module. 
\end{enumerate}
\end{definition}

The existence of a right $A$-module isomorphism means that there is a free generator, say $\omega$, of $\Omega^n A$ (as a right $A$-module), i.e. $\omega \in \Omega^n A$, such that all elements of $\Omega^n A$ can be uniquely expressed as $\omega a$ with $a \in A$. If $\omega$ is also a free generator of $\Omega^n A$ as a left $A$-module, this is said to be a {\em volume form} on $\Omega A$.

The right $A$-module isomorphism $\Omega^n A \to A$ corresponding to a volume form $\omega$ is denoted by $\pi_{\omega}$, i.e.
\begin{equation}\label{BrzezinskiSitarz2017(2.1)}
\pi_{\omega} (\omega a) = a, \quad {\rm for\ all}\ a\in A.
\end{equation}

By using that $\Omega^n A$ is also isomorphic to $A$ as a left $A$-module, any free generator $\omega $ induces an algebra endomorphism $\nu_{\omega}$ of $A$ by the formula
\begin{equation}\label{BrzezinskiSitarz2017(2.2)}
    a \omega = \omega \nu_{\omega} (a).
\end{equation}

Note that if $\omega$ is a volume form, then $\nu_{\omega}$ is an algebra automorphism.

Now, we proceed to recall the key ingredients of the {\em integral calculus} on $A$ as dual to its differential calculus. For more details, see Brzezinski et al. \cite{Brzezinski2008, BrzezinskiElKaoutitLomp2010}.

Let $(\Omega A, d)$ be a differential calculus on $A$. The space of $n$-forms $\Omega^n A$ is an $A$-bimodule. Consider $\mathcal{I}_{n}A$ the right dual of $\Omega^{n}A$, the space of all right $A$-linear maps $\Omega^{n}A\rightarrow A$, that is, $\mathcal{I}_{n}A := {\rm Hom}_{A}(\Omega^{n}(A),A)$. Notice that each of the $\mathcal{I}_{n}A$ is an $A$-bimodule with the actions
\begin{align*}
    (a\cdot\phi\cdot b)(\omega)=a\phi(b\omega),\quad {\rm for\ all}\ \phi \in \mathcal{I}_{n}A,\ \omega \in \Omega^{n}A\ {\rm and}\ a,b \in A.
\end{align*}

The direct sum of all the $\mathcal{I}_{n}A$, that is, $\mathcal{I}A = \bigoplus\limits_{n} \mathcal{I}_n A$, is a right $\Omega A$-module with action given by
\begin{align}\label{BrzezinskiSitarz2017(2.3)}
    (\phi\cdot\omega)(\omega')=\phi(\omega\wedge\omega'),\quad {\rm for\ all}\ \phi\in\mathcal{I}_{n + m}A, \ \omega\in \Omega^{n}A \ {\rm and} \ \omega' \in \Omega^{m}A.
\end{align}

\begin{definition}[{\cite[Definition 2.1]{Brzezinski2008}}]
A {\em divergence} (also called {\em hom-connection}) on $A$ is a linear map $\nabla: \mathcal{I}_1 A \to A$ such that
\begin{equation}\label{BrzezinskiSitarz2017(2.4)}
    \nabla(\phi \cdot a) = \nabla(\phi) a + \phi(da), \quad {\rm for\ all}\ \phi \in \mathcal{I}_1 A \ {\rm and} \ a \in A.
\end{equation}  
\end{definition}

Note that a divergence can be extended to the whole of $\mathcal{I}A$, 
\[
\nabla_n: \mathcal{I}_{n+1} A \to \mathcal{I}_{n} A,
\]

by considering
\begin{equation}\label{BrzezinskiSitarz2017(2.5)}
\nabla_n(\phi)(\omega) = \nabla(\phi \cdot \omega) + (-1)^{n+1} \phi(d \omega), \quad {\rm for\ all}\ \phi \in \mathcal{I}_{n+1}(A)\ {\rm and} \ \omega \in \Omega^n A.
\end{equation}

By putting together (\ref{BrzezinskiSitarz2017(2.4)}) and (\ref{BrzezinskiSitarz2017(2.5)}), we get the Leibniz rule 
\begin{equation}
    \nabla_n(\phi \cdot \omega) = \nabla_{m + n}(\phi) \cdot \omega + (-1)^{m + n} \phi \cdot d\omega,
\end{equation}

for all elements $\phi \in \mathcal{I}_{m + n + 1} A$ and $\omega \in \Omega^m A$ \cite[Lemma 3.2]{Brzezinski2008}. In the case $n = 0$, if ${\rm Hom}_A(A, M)$ is canonically identified with $M$, then $\nabla_0$ reduces to the classical Leibniz rule.

\begin{definition}[{\cite[Definition 3.4]{Brzezinski2008}}]
The right $A$-module map 
$$
F = \nabla_0 \circ \nabla_1: {\rm Hom}_A(\Omega^{2} A, M) \to M
$$ is called a {\em curvature} of a hom-connection $(M, \nabla_0)$. $(M, \nabla_0)$ is said to be {\em flat} if its curvature is the zero map, that is, if $\nabla \circ \nabla_1 = 0$. This condition implies that $\nabla_n \circ \nabla_{n+1} = 0$ for all $n\in \mathbb{N}$.
\end{definition}

$\mathcal{I} A$ together with the $\nabla_n$ form a chain complex called the {\em complex of integral forms} over $A$. The cokernel map of $\nabla$, that is, $\Lambda: A \to {\rm Coker} \nabla = A / {\rm Im} \nabla$ is said to be the {\em integral on $A$ associated to} $\mathcal{I}A$.

Given a left $A$-module $X$ with action $a\cdot x$, for all $a\in A,\ x \in X$, and an algebra automorphism $\nu$ of $A$, the notation $^{\nu}X$ stands for $X$ with the $A$-module structure twisted by $\nu$, i.e. with the $A$-action $a\otimes x \mapsto \nu(a)\cdot x $.

The following definition of an \textit{integrable differential calculus} seeks to portray a version of Hodge star isomorphisms between the complex of differential forms of a differentiable manifold and a complex of dual modules of it \cite[p. 112]{Brzezinski2015}. 

\begin{definition}[{\cite[Definition 2.1]{BrzezinskiSitarz2017}}]
An $n$-dimensional differential calculus $(\Omega A, d)$ is said to be {\em integrable} if $(\Omega A, d)$ admits a complex of integral forms $(\mathcal{I}A, \nabla)$ for which there exist an algebra automorphism $\nu$ of $A$ and $A$-bimodule isomorphisms \linebreak $\Theta_k: \Omega^{k} A \to ^{\nu} \mathcal{I}_{n-k}A$, $k = 0, \dotsc, n$, rendering commmutative the following diagram:
\[
\begin{tikzcd}
A \arrow{r}{d} \arrow{d}{\Theta_0} & \Omega^{1} A \arrow{d}{\Theta_1} \arrow{r}{d} & \Omega^2 A  \arrow{d}{\Theta_2} \arrow{r}{d} & \dotsb \arrow{r}{d} & \Omega^{n-1} A \arrow{d}{\Theta_{n-1}} \arrow{r}{d} & \Omega^n A  \arrow{d}{\Theta_n} \\ ^{\nu} \mathcal{I}_n A \arrow[swap]{r}{\nabla_{n-1}} & ^{\nu} \mathcal{I}_{n-1} A \arrow[swap]{r}{\nabla_{n-2}} & ^{\nu} \mathcal{I}_{n-2} A \arrow[swap]{r}{\nabla_{n-3}} & \dotsb \arrow[swap]{r}{\nabla_{1}} & ^{\nu} \mathcal{I}_{1} A \arrow[swap]{r}{\nabla} & ^{\nu} A
\end{tikzcd}
\]

The $n$-form $\omega:= \Theta_n^{-1}(1)\in \Omega^n A$ is called an {\em integrating volume form}. 
\end{definition}

The algebra of complex matrices $M_n(\mathbb{C})$ with the $n$-dimensional calculus generated by derivations presented by Dubois-Violette et al. \cite{DuboisViolette1988, DuboisVioletteKernerMadore1990}, the quantum group $SU_q(2)$ with the three-dimensional left covariant calculus developed by Woronowicz \cite{Woronowicz1987} and the quantum standard sphere with the restriction of the above calculus, are examples of algebras admitting integrable calculi. For more details on the subject, see Brzezi\'nski et al. \cite{BrzezinskiElKaoutitLomp2010}. 

The following proposition shows that the integrability of a differential calculus can be defined without explicit reference to integral forms. This allows us to guarantee the integrability by considering the existence of finitely generator elements that allow to determine left and right components of any homogeneous element of $\Omega(A)$.

\begin{proposition}[{\cite[Theorem 2.2]{BrzezinskiSitarz2017}}]\label{integrableequiva} 
Let $(\Omega A, d)$ be an $n$-dimensional differential calculus over an algebra $A$. The following assertions are equivalent:
\begin{enumerate}
    \item [\rm (1)] $(\Omega A, d)$ is an integrable differential calculus.
    
    \item [\rm (2)] There exists an algebra automorphism $\nu$ of $A$ and $A$-bimodule isomorphisms $\Theta_k : \Omega^k A \rightarrow \ ^{\nu}\mathcal{I}_{n-k}A$, $k =0, \ldots, n$, such that, for all $\omega'\in \Omega^k A$ and $\omega''\in \Omega^mA$,
    \begin{align*}
        \Theta_{k+m}(\omega'\wedge\omega'')=(-1)^{(n-1)m}\Theta_k(\omega')\cdot\omega''.
    \end{align*}
    
    \item [\rm (3)] There exists an algebra automorphism $\nu$ of $A$ and an $A$-bimodule map $\vartheta:\Omega^nA\rightarrow\ ^{\nu}A$ such that all left multiplication maps
    \begin{align*}
    \ell_{\vartheta}^{k}:\Omega^k A &\ \rightarrow \mathcal{I}_{n-k}A, \\
    \omega' &\ \mapsto \vartheta\cdot\omega', \quad k = 0, 1, \dotsc, n,
    \end{align*}
    where the actions $\cdot$ are defined by {\rm (}\ref{BrzezinskiSitarz2017(2.3)}{\rm )}, are bijective.
    
    \item [\rm (4)] $(\Omega A, d)$ has a volume form $\omega$ such that all left multiplication maps
    \begin{align*}
        \ell_{\pi_{\omega}}^{k}:\Omega^k A &\ \rightarrow \mathcal{I}_{n-k}A, \\
        \omega' &\ \mapsto \pi_{\omega} \cdot \omega', \quad k=0,1, \dotsc, n-1,
    \end{align*}
    
    where $\pi_{\omega}$ is defined by {\rm (}\ref{BrzezinskiSitarz2017(2.1)}{\rm )}, are bijective.
\end{enumerate}
\end{proposition}

A volume form $\omega\in \Omega^nA$ is an {\em integrating form} if and only if it satisfies condition $(4)$ of Proposition \ref{integrableequiva} \cite[Remark 2.3]{BrzezinskiSitarz2017}.

The most interesting cases of differential calculi are those where $\Omega^k A$ are finitely generated and projective right or left (or both) $A$-modules \cite{Brzezinski2011}.

\begin{proposition}\label{BrzezinskiSitarz2017Lemmas2.6and2.7}
\begin{enumerate}
\item [\rm (1)] \cite[Lemma 2.6]{BrzezinskiSitarz2017} Consider $(\Omega A, d)$ an integrable and $n$-dimensional calculus over $A$ with integrating form $\omega$. Then $\Omega^{k} A$ is a finitely generated projective right $A$-module if there exist a finite number of forms $\omega_i \in \Omega^{k} A$ and $\overline{\omega}_i \in \Omega^{n-k} A$ such that, for all $\omega' \in \Omega^{k} A$, we have that 
\begin{equation*}
\omega' = \sum_{i} \omega_i \pi_{\omega} (\overline{\omega}_i \wedge \omega').
\end{equation*}

\item [\rm (2)] \cite[Lemma 2.7]{BrzezinskiSitarz2017} Let $(\Omega A, d)$ be an $n$-dimensional calculus over $A$ admitting a volume form $\omega$. Assume that for all $k = 1, \ldots, n-1$, there exists a finite number of forms $\omega_{i}^{k},\overline{\omega}_{i}^{k} \in \Omega^{k}(A)$ such that for all $\omega'\in \Omega^kA$, we have that
\begin{equation*}
\omega'=\displaystyle\sum_i\omega_{i}^{k}\pi_\omega(\overline{\omega}_{i}^{n-k}\wedge\omega')=\displaystyle\sum_i\nu_{\omega}^{-1}(\pi_\omega(\omega'\wedge\omega_{i}^{n-k}))\overline{\omega}_{i}^{k},
\end{equation*}

where $\pi_{\omega}$ and $\nu_{\omega}$ are defined by {\rm (}\ref{BrzezinskiSitarz2017(2.1)}{\rm )} and {\rm (}\ref{BrzezinskiSitarz2017(2.2)}{\rm )}, respectively. Then $\omega$ is an integral form and all the $\Omega^{k}A$ are finitely generated and projective as left and right $A$-modules.
\end{enumerate}
\end{proposition}

Brzezi\'nski and Sitarz \cite[p. 421]{BrzezinskiSitarz2017} asserted that to connect the integrability of the differential graded algebra $(\Omega A, d)$ with the algebra $A$, it is necessary to relate the dimension of the differential calculus $\Omega A$ with that of $A$, and since we are dealing with algebras that are deformations of coordinate algebras of affine varieties, the {\em Gelfand-Kirillov dimension} introduced by Gelfand and Kirillov \cite{GelfandKirillov1966, GelfandKirillov1966b} seems to be the best suited. Briefly, given an affine $\Bbbk$-algebra $A$, the {\em Gelfand-Kirillov dimension of} $A$, denoted by ${\rm GKdim}(A)$, is given by
\[
{\rm GKdim}(A) := \underset{n\to \infty}{\rm lim\ sup} \frac{{\rm log}({\rm dim}\ V^{n})}{{\rm log}\ n},
\]

where $V$ is a finite-dimensional subspace of $A$ that generates $A$ as an algebra. This definition is independent of choice of $V$. If $A$ is not affine, then its Gelfand-Kirillov dimension is defined to be the supremum of the Gelfand-Kirillov dimensions of all affine subalgebras of $A$. An affine domain of Gelfand-Kirillov dimension zero is precisely a division ring that is finite-dimensional over its center. In the case of an affine domain of Gelfand-Kirillov dimension one over $\Bbbk$, this is precisely a finite module over its center, and thus polynomial identity. In some sense, this dimensions measures the deviation of the algebra $A$ from finite dimensionality. For more details about this dimension, see the excellent treatment developed by Krause and Lenagan \cite{KrauseLenagan2000}.

After preliminaries above, we arrive to the key notion of this paper.

\begin{definition}[{\cite[Definition 2.4]{BrzezinskiSitarz2017}}]\label{BrzezinskiSitarz2017Definition2.4}
An affine algebra $A$ with integer Gelfand-Kirillov dimension $n$ is said to be {\em differentially smooth} if it admits an $n$-dimensional connected integrable differential calculus $(\Omega A, d)$.
\end{definition}

Definition \ref{BrzezinskiSitarz2017Definition2.4} shows that a differentially smooth algebra comes equipped with a well-behaved differential structure and with the precise concept of integration \cite[p. 2414]{BrzezinskiLomp2018}.

\begin{example}\label{Brzezinski2015DSOEbiquadratic}
\begin{enumerate}
    \item [\rm (i)] The polynomial algebra $\Bbbk[x_1, \dotsc, x_n]$ has Gelfand-Kirillov dimension $n$ and the usual exterior algebra is an $n$-dimensional integrable calculus, whence $\Bbbk[x_1, \dotsc, x_n]$ is differentially smooth.

    \item [\rm (ii)] Brzezi{\'n}ski \cite{Brzezinski2015} characterized the differential smoothness of skew polynomial rings of the form $\Bbbk[t][x; \sigma_{q, r}, \delta_{p(t)}]$ where $\sigma_{q, r}(t) = qt + r$, with $q, r \in \Bbbk,\ q\neq 0$, and the $\sigma_{q, r}-$derivation $\delta_{p(t)}$ is defined as
\begin{equation}\label{deltap}
\delta_{p(t)} (f(t)) = \frac{f(\sigma_{q, r}(t)) - f(t)}{\sigma_{q, r}(t) - t} p(t),
\end{equation}

for an element $p(t) \in \Bbbk[t]$. $\delta_{p(t)}(f(t))$ is a suitable limit when $q = 1$ and $r = 0$, that is, when $\sigma_{q, r}$ is the identity map of $\Bbbk[t]$.

For the maps
\begin{equation}\label{Brzezinski2015(3.4)1}
\nu_t(t) = t,\quad \nu_t(x) = qx + p'(t)\quad {\rm and}\quad \nu_x(t) = \sigma_{q, r}^{-1}(t),\quad \nu_x(x) = x,
\end{equation}

where $p'(t)$ is the classical $t$-derivative of $p(t)$, Brzezi{\'n}ski \cite[Lemma 3.1]{Brzezinski2015} showed that all of them simultaneously extend to algebra automorphisms $\nu_t$ and $\nu_x$ of $\Bbbk[t][x; \sigma_{q, r}, \delta_{p(t)}]$ only in the following three cases:
    \begin{enumerate}
        \item [\rm (a)] $q = 1, r = 0$ with no restriction on $p(t)$;
        
        \item [\rm (b)] $q = 1, r\neq 0$ and $p(t) = c$, $c\in \Bbbk$;
        
        \item [\rm (c)] $q\neq 1, p(t) = c\left( t + \frac{r}{q-1} \right)$, $c\in \Bbbk$ with no restriction on $r$.
    \end{enumerate}
    
In any of the cases {\rm (a) - (c)} we have that $\nu_x \circ \nu_t = \nu_t \circ \nu_x$. If the Ore extension $\Bbbk[t][x; \sigma_{q, r}, \delta_{p(t)}]$ satisfies one of these three conditions, Brzezi{\'n}ski proved that it is differentially smooth \cite[Proposition 3.3]{Brzezinski2015}.

From Brzezi{\'n}ski's result we get that the algebras
\begin{itemize}
 \item The {\em polynomial algebra} $\Bbbk[x_1, x_2]$;
        
        \item The {\em Weyl algebra} $A_1(\Bbbk) = \Bbbk\{x_1, x_2\} / \langle x_1x_2 - x_2x_1 - 1\rangle$;
        
        \item The {\em universal enveloping algebra of the Lie algebra} $\mathfrak{n}_2 = \langle x_1, x_2\mid [x_2, x_1] = x_1\rangle$, that is, $U(\mathfrak{n}_2) = \Bbbk\{x_1, x_2\} / \langle x_2x_1 - x_1x_2 - x_1\rangle$, and
        
        \item The {\em quantum plane} ({\em Manin's plane}) $\mathcal{O}_q(\Bbbk) = \Bbbk \{x_1, x_2\} / \langle x_2 x_1 - qx_1 x_2\rangle$, where $q\in \Bbbk\ \backslash\ \{0,1\}$, and

\item {\em Jordan's plane} $\mathcal{J}(\Bbbk)$ with the relation given by $xt = tx + t^2$, 
\end{itemize}

are differentially smooth.
\end{enumerate}
\end{example}

\begin{remark}
There are examples of algebras that are not differentially smooth. Consider the commutative algebra $A = \mathbb{C}[x, y] / \langle xy \rangle$. A proof by contradiction shows that for this algebra there are no one-dimensional connected integrable calculi over $A$, so it cannot be differentially smooth \cite[Example 2.5]{BrzezinskiSitarz2017}.
\end{remark}

\section{Differential smoothness of SPBW extensions}\label{DSSPBW4originalresults}

\subsection{Extension of automorphisms and derivations }\label{ASPBWpreliminaries}

With the aim of setting up notation and render this paper self-contained, next we present the details of extended derivations over SPBW extensions on one (this is the case of the classical Ore extensions), two and $n$ generators (Propositions \ref{extaautoSPBW1}, \ref{extaautoSPBWn} and \ref{extaautoSPBWn}, respectively).

We start with the following well-known result (e.g. \cite[Section 2]{LamLeroyMatczuk1997}).

\begin{lemma}\label{commrule1}
Let $A = \sigma(R) \langle x \rangle$ be a SPBW extension over $R$, that is, $A = R[x; \sigma, \delta]$. For any $n \in \mathbb{N}$, we have that
\begin{equation*}
    x^nr = \sum_{k=0}^{n}\sum_{f_k\in C_{n-k,k}}f_k(r) x^{n-k},
\end{equation*}

where $C_{n-k,k}$ is the set of all possible compositions between $n-k$ $\sigma$'s and $k$ $\delta$'s. In addition, if $\sigma \circ \delta = \delta \circ \sigma$, we get that
\begin{equation*}
x^n r = \sum_{k=0}^{n}\binom{n}{k}\sigma^{n-k}\circ\delta^{k}(r) x^{n-k}.
\end{equation*}
\end{lemma}
\begin{proof}
We proceed by induction on $n$. For $n=1$, we know that $C_0=\{\sigma\}$ and $C_1=\{\delta\}$, whence
\begin{equation*}
    xr = \sigma(r)x + \delta(r) = \sum_{k=0}^{1}\sum_{f_k\in C_{n-k,k}}f_k(r) x^{1-k}.
\end{equation*}

Suppose the relation holds for $n$. Then
\begin{align*}
    x^{n+1} r = &\ x \left(\sum_{k=0}^{n}\sum_{f_k\in C_{n-k,k}}f_k(r) x^{n-k}\right) \\
    = &\ \sum_{k=0}^{n}\sum_{f_k\in C_{n-k,k}} x f_k(r) x^{n-k} = \sum_{k=0}^{n}\sum_{f_k\in C_{n-k,k}}(\sigma(f_k(r)) x + \delta(f_k(r))) x^{n-k} \\
    = &\ \sum_{k=0}^{n}\sum_{f_k\in C_{n-k,k}}\left[\sigma(f_k(r)) x^{n+1-k}+\delta(f_k(r)) x^{n-k}\right] \\
    = &\ \sum_{k=0}^{n+1}\sum_{g_k\in C_{n+1-k,k}}g_k(r) x^{n+1-k}.
\end{align*}

If $\sigma$ and $\delta$ commute, it follows that $C_k=\{\sigma^{n-k}\circ\delta^k\}$ and the number of times that the element is repeated is $\binom{n}{k}$, which implies that 
\begin{equation*}
    x^n r = \sum_{k=0}^{n}\sum_{f_k\in C_{n-k,k}}f_k(r) x^{n-k}=\sum_{k=0}^{n}\binom{n}{k}\sigma^{n-k}\circ\delta^{k}(r) x^{n-k}.
\end{equation*}
\end{proof}

The importance of extending endomorphisms and derivations of the ring $R$ (Proposition \ref{GallegoLezama2011Proposition3}) to the extension $A$ over $R$ can be appreciated in works such as \cite{NasrMoussavi2008, ReyesSuarez2017} for the study of ring-theoretical properties. In particular, the following result appears without proof in \cite[p. 514]{NasrMoussavi2008}.

\begin{proposition}\label{extaautoSPBW1}
Let $A = \sigma(R)\langle x \rangle$ be a SPBW extension over $R$ of automorphism type. Consider the automorphism $\widetilde{\sigma}:A \rightarrow A$ given by $\widetilde{\sigma}(x) = x$ and $\widetilde{\sigma}(r)=\sigma(r)$ for each $r\in R$. If  $\sigma \circ \delta = \delta \circ \sigma$, then the map $\widetilde{\delta}:A\rightarrow A$ with $\widetilde{\delta}(f(x))=\widetilde{\delta}(a_0+a_1x+ \cdots +a_nx^n)=\delta(a_0)+\delta(a_1)x+ \cdots + \delta(a_n)x^n$ for all $a_i \in R$, $0\leq i\leq n$ is a $\widetilde{\sigma}$-derivation of $A$.
\end{proposition}
\begin{proof}
    Let $p(x) = \sum\limits_{i=0}^n r_i x^i \in A $. Then
    \begin{align*}
        \widetilde{\sigma}(p(x)) = \widetilde{\sigma} \left(\sum_{i=0}^n r_i x^i\right) = \sum_{i=0}^n \widetilde{\sigma}(r_i x^i) = \sum_{i=0}^n\widetilde{\sigma} (r_i) x^i = \sum_{i=0}^n\sigma(r_i) x^i.
    \end{align*}
    
It is clear that the bijectivity of $\widetilde{\sigma}$ is inherited by $\sigma$.

Now, for $p(x) = r x^k$, we have that
\begin{equation*}
    \widetilde{\delta} (p(x)) = \widetilde{\delta}(r x^k) = \widetilde{\sigma}(r)\widetilde{\delta} (x^k) +x  \widetilde{\delta}(r) x^k = \sigma(r)\widetilde{\delta}(x^k) + \delta(r) x^k.
\end{equation*}

Let $r(x) = r x^i$ and $s(x) = s x^j$ for some $i,j \in \mathbb{N}$. The idea is to satisfy the relation given by
\begin{align*}
    \widetilde{\delta} (r(x) s(x)) = \widetilde{\sigma} (r(x)) \widetilde{\delta}(s(x)) + \widetilde{\delta} (r(x)) s(x).
\end{align*}

With this aim, note that 
\begin{align}
    \widetilde{\delta}(r(x) & s(x)) = \widetilde{\delta}(r x^is x^j) \notag \\ 
    = &\ \widetilde{\delta}\left(r\sum_{k=0}^{i}\sum_{f_k\in C_{i-k,k}}f_k(s) x^{i-k} x^j\right) \notag \\
     = &\ \widetilde{\delta}\left(\sum_{k=0}^{i}\sum_{f_k\in C_{i-k,k}} r f_k(s) x^{i-k+j}\right) \notag \\ 
     = &\ \sum_{k=0}^{i}\sum_{f_k\in C_{i-k,k}}\widetilde{\delta}\left(rf_k(s) x^{i-k+j}\right) \notag \\
     = &\ \sum_{k=0}^{i}\sum_{f_k\in C_{i-k,k}}\left(\sigma(rf_k(s))\widetilde{\delta} (x^{i-k+j}) + \delta(rf_k(s)) x^{i-k+j}\right) \notag \\
     = &\ \sum_{k=0}^{i}\sum_{f_k\in C_{i-k,k}}\left(\sigma(r)\sigma(f_k(s))\widetilde{\delta}(x^{i-k+j})+\left(\sigma(r)\delta(f_k(s))+\delta(r)f_k(s)\right) x^{i-k+j}\right) \notag 
     \end{align}
    \begin{align} 
     = &\ \sum_{k=0}^{i} \sum_{f_k\in C_{i-k,k}} \left(\sigma(r)\sigma(f_k(s))\left(x^{i-k}\widetilde{\delta}(x^j) + \widetilde{\delta}(x^{i-k}) x^j\right) \right. \notag \\
     &\ \ + \left. \left(\sigma(r)\delta(f_k(s)) + \delta(r)f_k(s)\right) x^{i-k+j} \right). \label{First24}
\end{align}

On the other hand,
{\small{
\begin{align}
    \widetilde{\sigma}(r(x)) \widetilde{\delta}(s(x)) + \widetilde{\delta}(r(x))s(x) = &\ \widetilde{\sigma}(r x^i)\widetilde{\delta}(s x^j)+\widetilde{\delta}(r x^i) sx^j \notag \\
    = &\ \sigma(r) x^i\left(\sigma(s)\widetilde{\delta}(x^j) + \delta(s) x^j\right) + \left(\sigma(r)\widetilde{\delta}(x^i) + \delta(r) x^i\right)s x^j \notag \\
    = &\ \sigma(r) x^i \sigma(s)\widetilde{\delta}(x^j) + \sigma(r) x^i\delta(s) x^j \\
    &\ + \sigma(r)\widetilde{\delta}(x^i) sx^j + \delta(r) x^is x^j \notag \\
    = &\ \sigma(r)\left(\sum_{k=0}^{i}\sum_{f_k\in C_{i-k,k}}f_k(\sigma(s)) x^{i-k}\right)\widetilde{\delta}(x^j) \\
    &\ + \sigma(r)\left(\sum_{k=0}^{i}\sum_{f_k\in C_{i-k,k}}f_k(\delta(s)) x^{i-k}\right) x^j \notag \\ &\ \ + \sigma(r)\widetilde{\delta}(x^i) s x^j + \delta(r) \left(\sum_{k=0}^{i}\sum_{f_k\in C_{i-k,k}}f_k(s) x^{i-k}\right)x^j \notag \\
    = &\ \sum_{k=0}^{i}\sum_{f_k\in C_{i-k,k}}\left(\sigma(r)f_k(\sigma(s)) x^{i-k}\widetilde{\delta}(x^j) + \sigma(r)f_k(\delta(s)) x^{i-k+j} \right. \\
    &\ + \left. \delta(r)f_k(s) x^{i-k+j}\right) + \sigma(r)\widetilde{\delta}(x^i)s x^j. \label{Second24}
\end{align}
}}

If we compare (\ref{First24}) and (\ref{Second24}), we get that
\begin{align*}
&\ \sum_{k=0}^{i}\sum_{f_k\in C_{i-k,k}}\left((\sigma(f_k(s))-f_k(\sigma(s))) x^{i-k}\widetilde{\delta}(x^j) + (\delta(f_k(s))-f_k(\delta(s))) x^{i-k+j} \right. \\
&\ \ + \left. \sigma(f_k(s))\widetilde{\delta}(x^{i-k}) x^j\right) - \widetilde{\delta}(x^i)s x^j=0.
\end{align*}

Let $\widetilde{\delta}(x^i) = \sum\limits_{l=0}^{n} a_lx^l$, $\widetilde{\delta}(x^j) = \sum\limits_{l=0}^{n}b_l x^l$ and $\widetilde{\delta}(x^{i-k}) = \sum\limits_{l=0}^{n}c_l x^l$. Then,
\begin{align*}
&\ \notag \sum_{k=0}^{i}\sum_{f_k\in C_{i-k,k}}\left((\sigma(f_k(s))-f_k(\sigma(s))) x^{i-k}\sum_{l=0}^{n}b_l x^l+(\delta(f_k(s))-f_k(\delta(s))) x^{i-k+j} \right. \\ \notag
&\ \ \left. + \ \sigma(f_k(s))\sum_{l=0}^{n}c_l x^l x^j\right) - \sum_{l=0}^{n}a_l x^ls x^j=0. \notag
\end{align*}

\begin{align*}
&\ \sum_{k=0}^{i}\sum_{f_k\in C_{i-k,k}}\left((\sigma(f_k(s))-f_k(\sigma(s)))\sum_{l=0}^{n}\left(\sum_{p=0}^{i-k}\sum_{h_p\in C_{i-k-p,p}}h_p(b_l) x^{i-k-p}\right) x^l \right. \\ \notag
&\ \ \left. + \ (\delta(f_k(s))-f_k(\delta(s))) x^{i-k+j} +\sigma(f_k(s))\sum_{l=0}^{n}c_l x^l x^j\right)\\ \notag
&\ -\sum_{l=0}^{n}a_l\left(\sum_{k=0}^l\sum_{g_k\in C_{l-k,k}}g_k(s) x^{l-k}\right) x^j = 0. 
\end{align*}

\begin{align}\label{Eq1}
&\ \sum_{k=0}^{i}\sum_{f_k\in C_{i-k,k}}\left((\sigma(f_k(s))-f_k(\sigma(s)))\sum_{l=0}^{n}\sum_{p=0}^{i-k}\sum_{h_p\in C_{i-k-p,p}}h_p(b_l)x^{i-k-p+l} \right. \\ \notag
&\ \ \left. + \ (\delta(f_k(s))-f_k(\delta(s)))x^{i-k+j} +\sigma(f_k(s))\sum_{l=0}^{n}c_l x^{l+j}\right)\\ \notag
&\ \ - \ \sum_{l=0}^{n}\sum_{k=0}^l\sum_{g_k\in C_{l-k,k}}a_lg_k(s)x^{l-k+j} = 0. 
\end{align}

When $k=0$, we put attention on the monomial $\sum_{f_0\in C_{i,0}}(\delta(f_0(s))-f_0(\delta(s)))x^{i+j}$. This is the monomial of the highest degree and since it is equal to zero, it follows that $\sum_{f_0\in C_{i,0}} \left(\delta(f_0(s))-f_0(\delta(s))\right) = 0$. Since the power $i$ is arbitrary, it can be seen that $\delta(\sigma(s))=\sigma(\delta(s))$. In the same way, the value of $s$ is arbitrary, which means that $\delta\sigma = \sigma\delta$.

Now, expression {\rm (}\ref{Eq1}{\rm )} can be written as
\begin{align*}
\sum_{k=0}^{i}\sum_{f_k\in C_{i-k,k}}\left(\sigma(f_k(s))\sum_{l=0}^{n}c_lt^{l+j}\right)-\sum_{l=0}^{n}\sum_{k=0}^l\sum_{g_k\in C_{l-k,k}}a_lg_k(s)x^{l-k+j} = &\ 0,
\end{align*}

or equivalently, 
\begin{align*}
\sum_{l=0}^{n}\left(\sum_{k=0}^{i}\sum_{f_k\in C_{i-k,k}}\sigma(f_k(s))c_lx^{l}-\sum_{k=0}^l\sum_{g_k\in C_{l-k,k}}a_lg_k(s)x^{l-k}\right) = &\ 0.
\end{align*}

Since this equality holds for all $s\in R$, any automorphism $\sigma$ and every $\sigma$-derivation $\delta$ of $R$, necessarily $a_l=c_l=0$ for all $0\leq l \leq n$. By changing the values of $i$ with $j$ we obtain that $b_l=0$. Then $\widetilde{\delta}(x)=0$. This means that 
$\widetilde{\delta}(f(x))=\widetilde{\delta}(a_0+a_1x+ \cdots +a_nx^n)=\delta(a_0)+\delta(a_1)x+ \cdots + \delta(a_n)x^n$.
\end{proof}

\begin{lemma}\label{commrule2}
Let $A = \sigma(R)\langle x_1, x_2 \rangle$ be a SPBW extension over $R$ such that $\sigma_i$ commutes with $\delta_i$ for $i = 1, 2$. Then for any $m \in \mathbb{N}$ we have that
\begin{equation*}
    x_i^mr = \sum_{k=0}^{m}\binom{m}{k}\sigma_i^{m-k}\circ\delta_i^{k}(r) x_i^{m-k}, \quad {\rm for} \ i = 1, 2.
\end{equation*}
\end{lemma}
\begin{proof}
This result can be proven in a completely analogous way to the Lemma \ref{commrule1}; it is enough to apply induction on $x_1$ and then on $x_2$. It must also be taken into account that $\sigma_i$ commutes with $\delta_i$ for $i = 1, 2$.
\end{proof}

\begin{proposition}\label{extaautoSPBW2}
Let $A=\sigma(R)\langle x_1, x_2 \rangle$ be a skew PBW extension over $R$ of automorphism type. Consider the endomorphisms $\widetilde{\sigma_i}:A \rightarrow A$ and $\widetilde{\delta_i}:A\rightarrow A$ defined as
\begin{align*}
    \widetilde{\sigma_i}\left(\sum r_kx_1^{\alpha_{1,k}}x_2^{\alpha_{2,k}}\right) &\ = \sum \sigma_i(r_k)x_1^{\alpha_{1,k}}x_2^{\alpha_{2,k}}, \\ \widetilde{\delta_i}\left(\sum r_kx_1^{\alpha_{1,k}}x_2^{\alpha_{2,k}}\right) &\ = \sum \delta_i(r_k)x_1^{\alpha_{1,k}}x_2^{\alpha_{2,k}},
\end{align*}
for $i= 1, 2$, and such that $\sigma_i$ and $\delta_i$ commute, $\delta_i\circ\delta_j=\delta_j\circ\delta_i$, $\delta_i\circ\sigma_j=\sigma_j\circ\delta_i$ and $\delta_k(d_{i,j})=\delta_k(r_l^{(i,j)})=0$ for $i, j, k, l= 1, 2$. Then $\widetilde{\sigma}$ is an automorphism of $A$ and $\widetilde{\delta}$ is a $\widetilde{\sigma_i}$-derivation of $A$.
\end{proposition}
\begin{proof}
Let $f = \displaystyle\sum_{k=0}^m r_kx_1^{\alpha_{1,k}} x_2^{\alpha_{2,k}}$, where $r_k\in R$ and $\alpha_{1,k}, \alpha_{2,k}\in\mathbb{N}$, for $1\leq k \leq m$. Then
\begin{align*}
   \widetilde{\sigma_i}(f) = \widetilde{\sigma_i}\left(\displaystyle\sum_{k=0}^m r_k x_1^{\alpha_{1,k}} x_2^{\alpha_{2,k}}\right) = \displaystyle\sum_{k=0}^m \widetilde{\sigma_i}\left(r_k x_1^{\alpha_{1,k}} x_2^{\alpha_{2,k}}\right) = \displaystyle\sum_{k=0}^m \widetilde{\sigma_i}\left(r_k\right) x_1^{\alpha_{1,k}} x_2^{\alpha_{2,k}}.
\end{align*}

The bijectivity of $\widetilde{\sigma_i}$ is inherited by $\sigma_i$ for $i=1, 2$.

Now, for $p=rx_1^{\alpha_{1}}x_2^{\alpha_{2}}$, we have that
\begin{align*}
     \widetilde{\delta_i}(p)&\ =\widetilde{\delta_i}(rx_1^{\alpha_{1}}x_2^{\alpha_{2}})=\widetilde{\sigma_i}(r)\widetilde{\delta_i}(x_1^{\alpha_{1}}x_2^{\alpha_{2}})+\widetilde{\delta_i}(r)x_1^{\alpha_{1}}x_2^{\alpha_{2}} \\
     &\ = \sigma_i(r)\left(\widetilde{\delta_i}(x_1^{\alpha_{1}})x_2^{\alpha_{2}}+ x_1^{\alpha_{1}}\widetilde{\delta_i}(x_2^{\alpha_{2}})\right)+\delta_i(r)x_1^{\alpha_{1}}x_2^{\alpha_{2}} \\
     &\ = \delta_i(r)x_1^{\alpha_{1}}x_2^{\alpha_{2}}.
\end{align*}

Consider $p=rx_1^{\alpha_{1}}x_2^{\alpha_{2}}$ and $s=sx_1^{\beta_{1}}x_2^{\beta_{2}}$ for some $\alpha_{1},\alpha_{2}, \beta_{1}, \beta_{2} \in \mathbb{N}$. Then,
{\footnotesize{
    \begin{align}
    \widetilde{\delta}_i(ps)= &\ \widetilde{\delta}(rx_1^{\alpha_{1}}x_2^{\alpha_{2}}sx_1^{\beta_{1}}x_2^{\beta_{2}}) \notag \\
    = &\ \widetilde{\delta}_i\left(rx_1^{\alpha_{1}}\sum_{k_2=0}^{\alpha_{2}}\binom{\alpha_{2}}{k_2}\sigma_2^{\alpha_{2}-k_2}\circ\delta_2^{k_2}(s)x_2^{\alpha_{2}-k_2}x_1^{\beta_{1}}x_2^{\beta_{2}}\right) \notag \\
     = &\ \widetilde{\delta}_i\left(\sum_{k_2=0}^{\alpha_{2}}\binom{\alpha_{2}}{k_2}rx_1^{\alpha_{1}}\sigma_2^{\alpha_{2}-k_2}\circ\delta_2^{k_2}(s)x_2^{\alpha_{2}-k_2}x_1^{\beta_{1}}x_2^{\beta_{2}} \right) \notag \\
     = &\ \widetilde{\delta}_i\left(\sum_{k_2=0}^{\alpha_{2}}\binom{\alpha_{2}}{k_2}r\sum_{k_1=0}^{\alpha_{1}}\binom{\alpha_{1}}{k_1}\sigma_1^{\alpha_{1}-k_1}\circ\delta_1^{k_1}\left(\sigma_2^{\alpha_{2}-k_2}\circ\delta_2^{k_2}(s)\right)x_1^{\alpha_{1}-k_1}x_2^{\alpha_{2}-k_2}x_1^{\beta_{1}}x_2^{\beta_{2}} \right) \notag \\
     = &\ \sum_{k_1=0}^{\alpha_{1}}\sum_{k_2=0}^{\alpha_{2}}\dbinom{\alpha_{1}}{k_1}\dbinom{\alpha_{2}}{k_2}\delta_i\left(r\sigma_1^{\alpha_{1}-k_1}\circ\delta_1^{k_1}\left(\sigma_2^{\alpha_{2}-k_2}\circ\delta_2^{k_2}(s)\right)\right)x_1^{\alpha_{1}-k_1}x_2^{\alpha_{2}-k_2}x_1^{\beta_{1}}x_2^{\beta_{2}} \label{reldelta12}.
    \end{align}
}}

On the other hand,
{\footnotesize{
    \begin{align}
\widetilde{\sigma}_i(p) & \widetilde{\delta}_i(s) + \widetilde{\delta}_i(p)s =  \widetilde{\sigma}_i(rx_1^{\alpha_{1}}x_2^{\alpha_{2}})\widetilde{\delta}_i(sx_1^{\beta_{1}}x_2^{\beta_{2}})+\widetilde{\delta}_i(rx_1^{\alpha_{1}}x_2^{\alpha_{2}})sx_1^{\beta_{1}}x_2^{\beta_{2}} \notag \\
    = &\ \sigma_i(r)x_1^{\alpha_{1}}x_2^{\alpha_{2}}\delta_i(s)x_1^{\beta_{1}}x_2^{\beta_{2}}+\delta_i(r)x_1^{\alpha_{1}}x_2^{\alpha_{2}}sx_1^{\beta_{1}}x_2^{\beta_{2}} \notag \\
    = &\ \sum_{k_1=0}^{\alpha_{1}}\sum_{k_2=0}^{\alpha_{2}}\binom{\alpha_{1}}{k_1}\binom{\alpha_{2}}{k_2}\sigma_i(r)\sigma_1^{\alpha_{1}-k_1}\circ\delta_1^{k_1}\left(\sigma_2^{\alpha_{2}-k_2}\circ\delta_2^{k_2}(\delta_i(s))\right)x_1^{\alpha_{1}-k_1}x_2^{\alpha_{2}-k_2}x_1^{\beta_{1}}x_2^{\beta_{2}} \notag \\
    &\ + \sum_{k_1=0}^{\alpha_{1}}\sum_{k_2=0}^{\alpha_{2}}\binom{\alpha_{1}}{k_1}\binom{\alpha_{2}}{k_2}\delta_i(r)\sigma_1^{\alpha_{1}-k_1}\circ\delta_1^{k_1}\left(\sigma_2^{\alpha_{2}-k_2}\circ\delta_2^{k_2}(s)\right)x_1^{\alpha_{1}-k_1}x_2^{\alpha_{2}-k_2}x_1^{\beta_{1}}x_2^{\beta_{2}} \label{reldelta22}.
    \end{align}
}}

We focus on the monomial $x_2^{\alpha_2-k_2}x_1^{\beta_1}$. With the aim of comparing the terms, necessarily the commutation rule of the ring must be applied: there, the elements $d_{1,2}$,  $r_k^{(1,2)}$, $0 \leq k \leq 2$, appear. As expected, these must be eliminated except $r_0^{(1,2)}$. For this reason, we assume that $\delta_k(d_{i,j})=\delta_k(r_l^{(i,j)})=0$ for $i, j, k, l= 1, 2$. If we compare {\rm (}\ref{reldelta12}{\rm )} with {\rm (}\ref{reldelta22}{\rm )}, then we get that 
\begin{align*}
    &\ \sigma_i(r)\delta_i\left(\sigma_1^{\alpha_{1}-k_1}\circ\delta_1^{k_1}\left(\sigma_2^{\alpha_{2}-k_2}\circ\delta_2^{k_2}(s)\right)\right)+\delta_i(r)\sigma_1^{\alpha_{1}-k_1}\circ\delta_1^{k_1}\left(\sigma_2^{\alpha_{2}-k_2}\circ\delta_2^{k_2}(s)\right) \\
    &\ = \sigma_i(r)\sigma_1^{\alpha_{1}-k_1}\circ\delta_1^{k_1}\left(\sigma_2^{\alpha_{2}-k_2}\circ\delta_2^{k_2}(\delta_i(s))\right) + \delta_i(r)\sigma_1^{\alpha_{1}-k_1}\circ\delta_1^{k_1}\left(\sigma_2^{\alpha_{2}-k_2}\circ\delta_2^{k_2}(s)\right) \\
\end{align*}
\begin{align*}
    &\ \sigma_i(r)\left(\delta_i\left(\sigma_1^{\alpha_{1}-k_1}\circ\delta_1^{k_1}\left(\sigma_2^{\alpha_{2}-k_2}\circ\delta_2^{k_2}(s)\right)\right)-\sigma_1^{\alpha_{1}-k_1}\circ\delta_1^{k_1}\left(\sigma_2^{\alpha_{2}-k_2}\circ\delta_2^{k_2}(\delta_i(s))\right)\right) \\
    &\ = 0.
\end{align*}

Since $\delta_i$ commute with $\sigma_j$ and $\delta_j$, for $j=1, 2$, this last equality holds, whence $\widetilde{\delta}_i$ is a $\widetilde{\sigma}_i$-derivation.
\end{proof}

\begin{lemma}
Let $A=\sigma(R)\langle x_1, \dotsc, x_n \rangle$ be a skew PBW extension over $R$ such that $\sigma_i$ commutes with $\delta_i$, $1\leq i \leq n$ for all $i$. Then, for any $m \in \mathbb{N}$ we have that
\begin{equation*}
x_i^mr=\sum_{k=0}^{m}\binom{m}{k}\sigma_i^{m-k}\circ\delta_i^{k}(r)x_i^{m-k}, \quad {\rm for\ all}\ 1\leq i\leq n.
\end{equation*}
\end{lemma}
\begin{proof}
The proof of this lemma, similar to the previous Lemmas \ref{commrule1} and \ref{commrule2}, is done by induction on each $x_i$, $1\leq i \leq n$. 
\end{proof}

\begin{proposition}\label{extaautoSPBWn}
Let $A=\sigma(R)\langle x_1,\ldots, x_n \rangle$ be a skew PBW extension over $R$ of automorphism type. Consider the endomorphisms $\widetilde{\sigma_i}:A \rightarrow A$ and $\widetilde{\delta_i}:A\rightarrow A$ defined as
\begin{align*}
    \widetilde{\sigma_i}\left(\sum r_kx_1^{\alpha_{1,k}}\cdots x_n^{\alpha_{n,k}}\right) &\ = \sum \sigma_i(r_k)x_1^{\alpha_{1,k}}\cdots x_n^{\alpha_{n,k}} \\
    \widetilde{\delta_i}\left(\sum r_kx_1^{\alpha_{1,k}}\cdots x_n^{\alpha_{n,k}}\right) &\ = \sum \delta_i(r_k)x_1^{\alpha_{1,k}}\cdots x_n^{\alpha_{n,k}},
\end{align*}
 for $1 \leq i \leq n$, and such that $\sigma_i$ and $\delta_i$ commute, $\delta_i\circ\delta_j=\delta_j\circ\delta_i$,  $\delta_i\circ\sigma_j=\sigma_j\circ\delta_i$ and $\delta_k(d_{i,j})=\delta_k(r_l^{(i,j)})=0$, for $1 \leq i, j, k, l \leq n$. Then $\widetilde{\sigma}$ is an automorphism of $A$ and $\widetilde{\delta}$ is a $\widetilde{\sigma_i}$-derivation of $A$.
\end{proposition}
\begin{proof}
Let $f = \displaystyle\sum_{k=0}^mc_kx_1^{\alpha_{1,k}}\cdots x_n^{\alpha_{n,k}}$, where $\alpha_{1,k}, \ldots, \alpha_{n,k}\in\mathbb{N}$, for $1\leq k \leq m$. Then
\begin{align*}
   \widetilde{\sigma_i}(a) &\ =\widetilde{\sigma_i}\left(\displaystyle\sum_{k=0}^mc_kx_1^{\alpha_{1,k}}\cdots x_n^{\alpha_{n,k}}\right) \\
   &\ =\displaystyle\sum_{k=0}^m\widetilde{\sigma_i}\left(c_kx_1^{\alpha_{1,k}}\cdots x_n^{\alpha_{n,k}}\right) \\
   &\ =\displaystyle\sum_{k=0}^m\widetilde{\sigma_i}\left(c_k\right)x_1^{\alpha_{1,k}}\cdots x_n^{\alpha_{n,k}}. 
\end{align*}

With the above, the bijectivity of $\widetilde{\sigma_i}$ is inherited by $\sigma_i$, for $1\leq i \leq n$.

For $p=rx_1^{\alpha_{1}}\cdots x_n^{\alpha_{n}}$, we have that
\begin{align*}
     \widetilde{\delta_i}(p)&\ =\widetilde{\delta_i}(rx_1^{\alpha_{1}}\cdots x_n^{\alpha_{n}}) \\ 
     &\ =\widetilde{\sigma_i}(r)\widetilde{\delta_i}(x_1^{\alpha_{1}}\cdots x_n^{\alpha_{n}})+\widetilde{\delta_i}(r)x_1^{\alpha_{1}}\cdots x_n^{\alpha_{n}} \\
     &\ = \sigma_i(r)\left(\widetilde{\delta_i}(x_1^{\alpha_{1}})x_2^{\alpha_{2}}\cdots x_n^{\alpha_{n}}+\cdots + x_1^{\alpha_{1}}\cdots\widetilde{\delta_i}(x_n^{\alpha_{n}})\right)+\delta_i(r)x_1^{\alpha_{1}}\cdots x_n^{\alpha_{n}} \\
     &\ = \delta_i(r)x_1^{\alpha_{1}}\cdots x_n^{\alpha_{n}}.
\end{align*}

Consider $p=rx_1^{\alpha_{1}}\cdots x_n^{\alpha_{n}}$ and $s=sx_1^{\beta_{1}}\cdots x_n^{\beta_{n}}$ for some $\alpha_{1},\ldots, \alpha_{n}\in \mathbb{N}$ and $\beta_{1}, \ldots, \beta_{n} \in \mathbb{N}$. Then
{\footnotesize{
    \begin{align*}
    \widetilde{\delta}_i(ps) = &\ \widetilde{\delta}(rx_1^{\alpha_{1}}\cdots x_n^{\alpha_{n}}sx_1^{\beta_{1}}\cdots x_n^{\beta_{n}}) \\
    = &\ \widetilde{\delta}_i\left(\sum_{k_1=0}^{\alpha_{1}}\cdots\sum_{k_n=0}^{\alpha_{n}}\left(\prod_{j=1}^{n}\binom{l_j}{k_j}\right)r\sigma_1^{\alpha_{1}-k_1}\circ \delta_1^{\alpha_{1}-k_1}\left(\cdots\sigma_{n}^{\alpha_{n}-k_n}\circ\delta_n^{k_n}(s)\cdots\right)x_1^{\alpha_{1}-k_1} \right. 
    \\
    &\ \left. \cdots \ x_{n-1}^{l_{n-1}-k_{n-1}}x_n^{\alpha_{n}-k_n}x_1^{\beta_{1}}\cdots x_n^{\beta_{n}}\right) \\
     = &\ \sum_{k_1=0}^{\alpha_{1}}\cdots\sum_{k_n=0}^{\alpha_{n}}\left(\prod_{j=1}^{n}\binom{l_j}{k_j}\right)\delta_i\left(r\sigma_1^{\alpha_{1}-k_1}\circ \delta_1^{\alpha_{1}-k_1}\left(\cdots\sigma_{n}^{\alpha_{n}-k_n}\circ\delta_n^{k_n}(s)\cdots\right)\right)x_1^{\alpha_{1}-k_1} \\
     &\ \cdots x_{n-1}^{l_{n-1}-k_{n-1}}x_n^{\alpha_{n}-k_n}x_1^{\beta_{1}}\cdots x_n^{\beta_{n}}. 
    \end{align*}
}}

On the other hand,
{\footnotesize{
    \begin{align*}
    \widetilde{\sigma}_i(p)\widetilde{\delta}_i(s) & + \widetilde{\delta}_i(p)s = \widetilde{\sigma}_i(rx_1^{\alpha_{1}}\cdots x_n^{\alpha_{n}})\widetilde{\delta}_i(sx_1^{\beta_{1}}\cdots x_n^{\beta_{n}})+\widetilde{\delta}_i(rx_1^{\alpha_{1}}\cdots x_n^{\alpha_{n}})sx_1^{\beta_{1}}\cdots x_n^{\beta_{n}} \\
    = &\ \sigma_i(r)x_1^{\alpha_{1}}\cdots x_n^{\alpha_{n}}\delta_i(s)x_1^{\beta_{1}}\cdots x_n^{\beta_{n}}+\delta_i(r)x_1^{\alpha_{1}}\cdots x_n^{\alpha_{n}}sx_1^{\beta_{1}}\cdots x_n^{\beta_{n}} \\
    = &\ \sum_{k_1=0}^{\alpha_{1}}\cdots\sum_{k_n=0}^{\alpha_{n}}\left(\prod_{j=1}^{n}\binom{l_j}{k_j}\right)\sigma_i(r)\sigma_1^{\alpha_{1}-k_1}\circ \delta_1^{\alpha_{1}-k_1}\left(\cdots\sigma_{n}^{\alpha_{n}-k_n}\circ\delta_n^{k_n}(\delta_i(s))\cdots\right) \\
    &\ x_1^{\alpha_{1}-k_1}\cdots x_{n-1}^{l_{n-1}-k_{n-1}}x_n^{\alpha_{n}-k_n}x_1^{\beta_{1}}\cdots x_n^{\beta_{n}} \\
    &\ + \sum_{k_1=0}^{\alpha_{1}}\cdots\sum_{k_n=0}^{\alpha_{n}}\left(\prod_{j=1}^{n}\binom{l_j}{k_j}\right)\delta_i(r)\sigma_1^{\alpha_{1}-k_1}\circ \delta_1^{\alpha_{1}-k_1}\left(\cdots\sigma_{n}^{\alpha_{n}-k_n}\circ\delta_n^{k_n}(s)\cdots\right) \\
    &\ x_1^{\alpha_{1}-k_1}\cdots x_{n-1}^{l_{n-1}-k_{n-1}}x_n^{\alpha_{n}-k_n}x_1^{\beta_{1}}\cdots x_n^{\beta_{n}}.
    \end{align*}
}}

For the same reason as in Proposition \ref{extaautoSPBW2}, when the elements $d_{i,j}$, $r_k^{(i,j)}$, $1\leq i < j \leq n$, $0 \leq k \leq n$, appear, these must cancel out when the terms are equalized, except $r_0^{(i,j)}$, $1\leq i < j \leq n$. For this reason, we assume that $\delta_k(d_{i,j})=\delta_k(r_l^{(i,j)})=0$ for $0\leq i, j, k, l \leq n$. In this way, 
\begin{align*}
    &\ \sigma_i(r)\delta_i\left(\sigma_1^{\alpha_{1}-k_1}\circ \delta_1^{\alpha_{1}-k_1}\left(\cdots\sigma_{n}^{\alpha_{n}-k_n}\circ\delta_n^{k_n}(s)\cdots\right)\right) \\
    &\ \quad  + \delta_i(r)\sigma_1^{\alpha_{1}-k_1}\circ \delta_1^{\alpha_{1}-k_1}\left(\cdots\sigma_{n}^{\alpha_{n}-k_n}\circ\delta_n^{k_n}(s)\cdots\right) \\
    &\ =\sigma_i(r)\sigma_1^{\alpha_{1}-k_1}\circ \delta_1^{\alpha_{1}-k_1}\left(\cdots\sigma_{n}^{\alpha_{n}-k_n}\circ\delta_n^{k_n}(\delta_i(s))\cdots\right) \\
    &\ \quad + \delta_i(r)\sigma_1^{\alpha_{1}-k_1}\circ \delta_1^{\alpha_{1}-k_1}\left(\cdots\sigma_{n}^{\alpha_{n}-k_n}\circ\delta_n^{k_n}(s)\cdots\right) \\
    &\ = \sigma_i(r)\left(\delta_i\left(\sigma_1^{\alpha_{1}-k_1}\circ \delta_1^{\alpha_{1}-k_1}\left(\cdots\sigma_{n}^{\alpha_{n}-k_n}\circ\delta_n^{k_n}(s)\cdots\right)\right) \right. \\
    &\ \quad - \left. \sigma_1^{\alpha_{1}-k_1}\circ \delta_1^{\alpha_{1}-k_1}\left(\cdots\sigma_{n}^{\alpha_{n}-k_n}\circ\delta_n^{k_n}(\delta_i(s))\cdots\right)\right) \\
    &\ =  0.
\end{align*}

By using that $\delta_i$ commute with $\sigma_j$ and $\delta_j$, for $1\leq j \leq n$, the previous equation holds. Therefore, $\widetilde{\delta}_i$ is a $\widetilde{\sigma}_i$-derivation.
\end{proof}

\begin{remark}
    The commutativity of the maps $\sigma$'s and $\delta$'s, that is, of the system of endomorphisms $\Sigma$ and $\Sigma$-derivations $\Delta$ of $R$, as in Proposition \ref{extaautoSPBWn}, was considered by Lezama et al. \cite{LezamaAcostaReyes2015} to characterize prime ideals of SPBW extensions.
\end{remark}

\subsection{Differential smoothness}\label{DifferentialsmoothnesSPBWEAD}

Since we are considering SPBW extensions over any associative and unital ring $R$, we need some natural conditions to compute the Gelfand-Kirillov dimension of a skew PBW extension $A$ over $R$. Having in mind that Lezama and Venegas \cite{LezamaVenegas2020} generalized the classical notion of Gelfand-Kirillov dimension \cite{GelfandKirillov1966, GelfandKirillov1966b} considered by Brzezi\'nski in his definition of differential smoothness (Section \ref{DefinitionsandpreliminariesDSA}), throughout this section we follow this more general setting: we assume that $R$ is an $S$-algebra ($S$ a commutative domain) with a generator frame $V$ and the {\em rank} of $R$ is understood as $\text{rank}(R)=\text{dim}_Q(Q \otimes R) < \infty$ where $Q$ is the field of fractions of $S$ (see \cite[Theorem 2.1]{LezamaVenegas2020} for more details).

The following theorem is the most important result of the paper.

\begin{theorem}\label{smoothSPBWn}
Let $A = \sigma(R)\langle x_1,\ldots, x_n \rangle$ be a bijective skew PBW extension such that $\sigma_i$ and $\delta_i$ are $S$-linear, $\sigma_i(V)\subseteq V$ for all $1\leq i \leq n$, and $d_{i,j}=1$, $r_k^{(i,j)}=0$ for $1\leq k\leq n$, $1\leq i,j, \leq n$ {\rm (}Definition \ref{defpbwextension} {\rm (}iv{\rm )}{\rm )}. Consider $\widetilde{\sigma}_i$ and $\widetilde{\delta}_i$ as in Proposition \ref{extaautoSPBWn} for $1\leq i \leq n$. If $\sigma_i \circ \sigma_j = \sigma_j \circ \sigma_i$ for all $1\leq i < j \leq n$, then $A$ is differentially smooth.
\end{theorem}
\begin{proof}
From \cite[Theorem 2.1]{LezamaVenegas2020}, we get that ${\rm GKdim}(\sigma(R)\langle x_1,\ldots, x_n\rangle) = n$. Hence, we proceeed to construct an $n$-dimensional integrable calculus. 

Consider $\Omega^{1}(\sigma(R)\langle x_1,\ldots, x_n\rangle)$ a free right $\sigma(R)\langle x_1,\ldots, x_n\rangle$-module of rank $n$ with generators $dx_1, \ldots, dx_n$. Define a left $\sigma(R)\langle x_1,\ldots, x_n\rangle$-module structure by
    \begin{equation}\label{relrightmodn}
        f dx_i = dx_i\widetilde{\sigma}_{i}(f),  \quad {\rm for\ all}\ 1\leq i \leq n, \  f\in \sigma(R)\langle x_1,\ldots, x_n\rangle.
    \end{equation}

Notice that the relations in $\Omega^{1}(\sigma(R)\langle x_1,\ldots, x_n\rangle)$ are given by
    \begin{align}\label{reltdt}
        x_idx_j =  dx_jx_i, \quad {\rm for\ all}\ 1\leq i < j \leq n.     
    \end{align}

We want to extend the assignments $x_i\mapsto dx_i$ and $1\leq i \leq n$ to a map 
$$
d: \sigma(R)\langle x_1,\ldots, x_n\rangle \to \Omega^{1}(\sigma(R)\langle x_1,\ldots, x_n\rangle)
$$

satisfying the Leibniz's rule. As expected, this is possible if we guarantee its compatibility with the non-trivial relations {\rm (}\ref{relSPBW}{\rm )}, i.e. if
    \begin{align*}
        dx_jx_i+x_jdx_i = d_{i,j}dx_ix_j+d_{i,j}x_idx_j +\sum_{k=1}^{n}r_{k}^{(i,j)}dx_k, \quad {\rm for}\ i < j.
    \end{align*}

Define $S$-linear maps 
$$
\partial_{x_i}: \sigma(R)\langle x_1,\ldots, x_n\rangle \rightarrow \sigma(R)\langle x_1,\ldots, x_n\rangle
$$ 

such that
\begin{align*}
    d(f) = \sum_{i=1}^{n}dx_i \partial_{x_i}(f), \quad {\rm for\ all}\ f \in \sigma(R)\langle x_1,\ldots, x_n\rangle.
\end{align*}

These maps are well-defined since the elements $dx_i$ for $1\leq i \leq n$ are free generators of the right $\sigma(R)\langle x_1,\ldots, x_n\rangle$-module $\Omega^1(\sigma(R)\langle x_1,\ldots, x_n\rangle)$. In this way, $d(f) = 0$ if and only if $\partial_{x_i}(f) = 0$ for $1 \leq i \leq n$. Using relations {\rm (}\ref{relrightmodn}{\rm )} and definitions of the maps $\widetilde{\sigma}_{i}$ for $1 \leq i \leq n$, we find that
\begin{align}
    \partial_{x_i}(x_1^{\alpha_{1}}\cdots x_n^{\alpha_{n}}) = &\ l_{i}x_{1}^{\alpha_{1}}\cdots x_i^{l_i-1}x_{i+1}^{l_{i+1}}\cdots x_n^{\alpha_{n}}. \notag
\end{align}

Then $d(f) = 0$ if and only if $f$ is a scalar multiple of the identity. This shows that $(\Omega(\sigma(R)\langle x_1,\ldots, x_n\rangle), d)$ is connected, where 
\[
\Omega (\sigma(R)\langle x_1,\ldots, x_n\rangle) = \bigoplus_{i=0}^{n}\Omega^i\left(\sigma(R)\langle x_1,\ldots, x_n\rangle\right).
\]

The universal extension of $d$ to higher forms compatible with {\rm (}\ref{reltdt}{\rm )} gives the following rules for $\Omega^l(\sigma(R)\langle x_1,\ldots, x_n\rangle)$ for $2 \le l \leq n-1$, 
\begin{align}\label{relnwedge}
    \bigwedge_{k=1}^{l}dx_{q(k)} = &\ (-1)^{\sharp}\bigwedge_{k=1}^ldx_{p(k)}, 
\end{align}

where $q:\{1,\ldots,l\}\rightarrow \{1,\ldots,n\}$ is an injective map, $p:\{1,\ldots,l\}\rightarrow \text{Im}(q)$ is an increasing injective map and $\sharp$ is the number of $2$-permutation needed to transform $q$ into $p$.

By assumption the automorphisms $\widetilde{\sigma}_{i}$'s commute with each other, which implies that there are no additional relations to the previous ones. Then
\begin{align*}
 \Omega^{n-1}(\sigma(R)\langle x_1,\ldots, x_n\rangle) = &\ \left[dx_1\wedge dx_2\wedge \cdots \wedge dx_{n-1}\oplus dx_1\wedge dx_3\wedge \cdots \wedge dx_{n} \right. \\
    &\ \left. \oplus \cdots \oplus \ dx_2\wedge \cdots \wedge dx_{n}\right]\sigma(R)\langle x_1,\ldots, x_n\rangle. 
\end{align*}

Since 
\[
\Omega^{n} (\sigma(R)\langle x_1,\ldots, x_n\rangle) = \omega\sigma(R)\langle x_1,\ldots, x_n\rangle\cong \sigma(R)\langle x_1,\ldots, x_n\rangle
\] 

as a right and left $\sigma(R)\langle x_1,\ldots, x_n\rangle$-module, with $\omega=dx_1 \wedge \cdots \wedge dx_n$, where $\widetilde{\sigma}_{\omega}=\widetilde{\sigma}_{1}\circ\cdots\circ\widetilde{\sigma}_{n}$, it follows that $\omega$ is a volume form of the SPBW extension $\sigma(R)\langle x_1,\ldots, x_n\rangle$. From Proposition \ref{BrzezinskiSitarz2017Lemmas2.6and2.7} (2), we get that $\omega$ is an integral form by setting
\begin{align*}
\omega_i^j = &\ \bigwedge_{k=1}^{j}dx_{p_{i,j}(k)}, \quad \text{ for } 1\leq i \leq \binom{n}{j}, \quad {\rm and} \\
\bar{\omega}_i^{n-j} = &\ (-1)^{\sharp_i}\bigwedge_{k=j+1}^{n}dx_{\bar{p}_{i,j}(k)},  \quad \text{for}\ 1\leq i \leq \binom{n}{j},
\end{align*}

where $p_{i,j}:\{1,\ldots,j\}\rightarrow \{1,\ldots,n\}$ is an increasing injective map, $\bar{p}_{i,j}:\{j+1,\ldots,n\}\rightarrow (\text{Im}(p_{i,j}))^c$ is also an increasing injective map and $\sharp_{i,j}$ is the number of $2$-permutation needed to transform $\{\bar{p}_{i,j}(j+1),\ldots, \bar{p}_{i,j}(n),p_{i,j}(1), \ldots, p_{i,j}(j)\}$ into $\{1, \ldots, n\}$. 

Let $\omega' \in \Omega^j(\sigma(R)\langle x_1,\ldots, x_n\rangle)$. Then:
\[
\omega' = \displaystyle\sum_{i=1}^{\binom{n}{j}}\bigwedge_{k=1}^{j}dx_{p_{i,j}(k)}a_i, \quad {\rm with}\ a_i \in R.
\] 

This implies that we have the equalities given by
{\footnotesize{
    \begin{align*}
    \sum_{i=1}^{\binom{n}{j}}\omega_{i}^{j}\pi_{\omega}(\bar{\omega}_i^{n-j}\wedge \omega') &\ =\sum_{i=1}^{\binom{n}{j}}\bigwedge_{k=1}^{j}dx_{p_{i,j}(k)}\pi_{\omega}\left(a_i(-1)^{\sharp_{i,j}}\bigwedge_{k=j+1}^{n}dx_{\bar{p}_{i,j}(k)}\wedge\bigwedge_{k=1}^{j}dx_{p_{i,j}(k)}\right) \\
    &\ = \sum_{i=1}^{\binom{n}{j}}\bigwedge_{k=1}^{j}dx_{p_{i,j}(k)}a_i \\ 
    &\ = \omega', 
    \end{align*}
}}
and finally, by Proposition \ref{BrzezinskiSitarz2017Lemmas2.6and2.7} (2), we conclude that $\sigma(R)\langle x_1,\ldots, x_n\rangle$ is differentially smooth.
\end{proof}

\begin{example}
Consider the skew polynomial ring $R[x; \sigma, \delta]$ studied by Nasr-Isfahani and Moussavi \cite{NasrMoussavi2008}. As we saw in the Introduction, $\sigma$ is an automorphism of $R$, $\delta$ is a $\sigma$-derivation such that $\alpha \delta = \delta \alpha$, and the extended automorphism $\overline{\sigma}$ and the $\overline{\sigma}$-derivation $\overline{\delta}$ on $R[x; \sigma, \delta]$ are given by
\begin{align*}
    \overline{\sigma}(f(x)) = &\ \sigma(r_0) + \sigma(r_1) x + \dotsb + \sigma(r_n)x^n, \quad {\rm and} \\
     \overline{\delta}(f(x)) = &\ \delta(r_0) + \delta(r_1) x + \dotsb + \delta(r_n)x^n, 
\end{align*}

respectively. As it is clear, these assumptions satisfy those corresponding in Proposition \ref{extaautoSPBWn}, and hence Theorem \ref{smoothSPBWn} shows that $R[x; \sigma, \delta]$ is differentially smooth when $\text{GKdim}(R)=0$.

By using a similar reasoning, and under the natural assumptions, iterated Ore extensions $R[x_1; \sigma_1, \delta_1][x_2;\sigma_2, \delta_2] \dotsc R[x_1; \sigma_n, \delta_n]$ are also differentially smooth.
\end{example}

\begin{example}
Artamonov et al. \cite{ArtamonovLezamaFajardo2016} studied extended modules, Vaserstein's, Quillen's patching, Horrock's, and Quillen-Suslin's theorems for a special class of Ore extensions. They assumed that for a ring $R$, $A$ denotes the Ore extension $A := R[x_1, \dotsc, x_n; \sigma]$ for which $\sigma$ is an automorphism of $R$, $x_i x_j = x_j x_i$ and $x_i r = \sigma_i(r) x_i$, for every $1\le i, k \le n$. As it is clear, this kind of Ore extensions satisfies the assumptions in Theorem \ref{smoothSPBWn}.
\end{example}

\begin{remark}\label{Brzezinski2015ExamplesOE}
Related to Ore extensions, and under some mild and geometrically natural assumptions, Brzezi\'nski and Lomp proved that if $R$ and $S$ are algebras with integrable calculi $(\Omega R, d_R)$ and $(\Omega S, d_S)$, $\Omega R$ is a finitely generated projective right $R$-module and $\Omega$ is a finitely generated projective right $S$-module, then $(\Omega R \otimes \Omega S, d)$ is an integrable differential calculus for $R\otimes S$ \cite[Proposition 3.1]{BrzezinskiLomp2018}. With this, if $R$ and $S$ are differentially smooth algebras with respect to calculi which are finitely generated projective as right modules ${\rm GKdim}(R \otimes S) = {\rm GKdim}(R) + {\rm GKdim}(S)$, then the tensor product algebra $R\otimes S$ is differentially smooth \cite[Corollary 3.2]{BrzezinskiLomp2018}.

They also proved that if $R$ is an algebra with an integrable differential calculus $(\Omega R, d)$ such that $\Omega R$ is a finitely generated right $R$-module, for any automorphism $\sigma$ of $R$ that extends to a degree-preserving automorphism of $\Omega R$, which commutes with $d$, there exists an integrable differential calculus $(\Omega A, d)$ on the skew polynomial ring $R[x;\sigma]$ and the Laurent skew polynomial ring $R[x^{\pm 1}; \sigma]$. If $R$ is differentially smooth with respect to $(\Omega R, d)$ and ${\rm GKdim}(A) = {\rm GKdim}(R) + 1$, then $A$ is also differentially smooth \cite[Theorem 4.1]{BrzezinskiLomp2018}. The following examples illustrates this situation.

For any non-zero $q\in \Bbbk^{*}$, let $A_q$ be the algebra generated by the indeterminates $x, y, z$ and relations $xy = yx,\ xz = qzy, \ yz = zx$, and let $B_q$ be the algebra generated by $x, y$ and invertible $z$ subject to the same relations. Since $A_q = \Bbbk[x, y][z;\sigma]$ and $B_q = \Bbbk[x, y][z^{\pm 1};\sigma]$ with $\sigma(x) = y$ and $\sigma(y) = qx$, then $A_q$ and $B_q$ are differentially smooth \cite[Example 4.6]{BrzezinskiLomp2018}. Let us see the details.

Note that the polynomial algebra $\Bbbk[x, y]$ is differentially smooth with the usual commutative differential calculus $\Omega (\Bbbk[x, y])$, i.e.,
\begin{align*}
       x dx = &\ dx x, & x dy = &\ dy x, & y dx = &\ dx y, \\ 
       y dy = &\ dy y, & dx dy = &\ - dy dx, & (dx)^2 = &\ (dy)^2 = 0. 
\end{align*}

The automorphism $\sigma$ extends to an automorphism of $\Omega (\Bbbk[x, y])$ by requesting it commute with $d$, that is, $\sigma(dx) = dy$ and $\sigma(dy) = q dx$.

Note that $\Omega (\Bbbk[x, y])$ is finitely generated as a right $\Bbbk[x, y]$-module and
\[
{\rm GKdim}(A_q) = {\rm GKdim}(B_q) = 3 = {\rm GKdim}(\Bbbk[x, y]) + 1,
\]

whence $A_q$ and $B_q$ are differentially smooth.

By using a similar reasoning \cite[Corollary 4.9]{BrzezinskiLomp2018}, it can be seen that the coordinate ring of the so called {\em quantum affine} $n$-{\em space}, that is, the algebra generated by the indeterminates $x_1, \dotsc, x_n$ subject to the relations $x_j x_i = q_{ij} x_i x_j$ for all $i < j$ with $q_{ij} \in \Bbbk^{*}$, is differentially smooth (c.f. \cite[Corollary 6 and Theorem 9]{KaracuhaLomp2014}).
\end{remark}

The results appearing in Remark \ref{Brzezinski2015ExamplesOE} motivate us to formulate the following question:

\vspace{0.3cm}

\noindent {\bf Question.} Let $A$ be a skew PBW extension over a differential smooth algebra $R$. Under which conditions does the differential smooth property pass from $R$ to $A$?

\end{document}